\documentclass[12pt]{amsart}
\usepackage{amsfonts,amssymb,amsthm,eucal,amsmath}
\usepackage{amscd}
\usepackage{graphicx}
\usepackage[T1]{fontenc}
\usepackage{latexsym,url}
\usepackage{pinlabel}
\usepackage{microtype}
\usepackage{hyperref}
\usepackage[inner=30mm, outer=30mm, textheight=225mm]{geometry}

\usepackage{subfig}
\usepackage{caption}
\newtheorem{thm}{Theorem}[section]
\newtheorem{lemma}[thm]{Lemma}

\newtheorem{cor}[thm]{Corollary}
\theoremstyle{definition}
\newtheorem{defn}[thm]{Definition}

\newtheorem{rmk}[thm]{Remark}

\newcommand{\cT}{\mathcal{T}}
\newcommand{\tri}{\bigtriangleup}

\title[Connectivity of triangulations without degree one edges]{Connectivity of triangulations without degree one edges under 2-3 and 3-2 moves}
\author{Henry Segerman}
\thanks{The author was supported in part by National Science Foundation grant DMS-1308767.}

\begin{document}

\begin{abstract}
Matveev and Piergallini independently showed that, with a small number of known exceptions, any triangulation of a three-manifold can be transformed into any other triangulation of the same three-manifold with the same number of vertices, via a sequence of 2-3 and 3-2 moves. We can interpret this as showing that the Pachner graph of such triangulations is connected. In this paper, we extend this result to show that (again with a small number of known exceptions), the subgraph of the Pachner graph consisting of triangulations without degree one edges is also connected, for single-vertex triangulations of closed manifolds, and ideal triangulations of manifolds with non-spherical boundary components.
\end{abstract}

\maketitle

\section{Introduction}

Matveev~\cite{matveev_2-3_connectivity_paper, matveev} and (independently) Piergallini~\cite{piergallini} show that the set of triangulations of a  three-manifold is connected under 2-3 and 3-2 moves, as long as we ignore the small number of triangulations of manifolds that consist of only a single tetrahedron. For these, no 2-3 or 3-2 move can be applied. Said another way, we consider the \emph{Pachner graph} $\mathfrak{T}(M)$, whose vertices are the triangulations of a given manifold $M$, where two vertices are connected by an edge if the corresponding triangulations are related by a 2-3 move. Matveev's and Piergallini's result then says that $\mathfrak{T}(M)$ is connected, again if we ignore the one-tetrahedron triangulations.

A natural question to ask is whether this remains true when we impose conditions on the triangulations. In other words we consider the connectivity of the subgraph of $\mathfrak{T}(M)$ corresponding to a given property. For example, we could consider the subgraph of geometric triangulations. Hoffman noted that the subgraph of geometric triangulations of the figure 8 knot complement is not connected; in fact there are isolated geometric triangulations (this observation is mentioned in \cite{dadd_duan}). However, geometricity is a very strong property. There are many weaker properties, corresponding to larger subgraphs, which may be connected. Of particular interest are the properties of \emph{1-efficiency} and of \emph{having essential edges}. The \emph{3D index} is a recent quantum object mapping ideal triangulations to Laurent series, introduced by Dimofte, Gukov, and Gaiotto \cite{dimofte_gaiotto_gukov_1, dimofte_gaiotto_gukov_2}. Garoufalidis \cite{garoufalidis_3D_index} showed that the 3D index is invariant under 2-3 and 3-2 moves, provided that it is defined on both sides of the move. In \cite{ghrs_index_structures}, Garoufalidis, Hodgson, Rubinstein and the author show that the 3D index is defined only on 1-efficient triangulations. Therefore, if we knew that the subgraph of $\mathfrak{T}(M)$ corresponding to 1-efficient triangulations were connected, we would have an invariant of the manifold. Unfortunately this is not known. Also in \cite{ghrs_index_structures}, we define an extremely small subgraph of the 1-efficient triangulations that we can prove is connected, and since the subgraph is determined purely by the topology of the manifold, we have an invariant. However, this is a somewhat unsatisfying workaround. A very similar story holds for the \emph{1-loop invariant}, as defined by Dimofte and Garoufalidis \cite{dimofte_garoufalidis}. In this case, the triangulations are required to have solutions to Thurston's gluing equations corresponding to the complete hyperbolic structure, which is implied by the triangulation having essential edges. Again it is not known if the triangulations with essential edges form a connected subgraph of $\mathfrak{T}(M)$.

In this paper, we answer the connectivity question in the affirmative for the subgraph of triangulations with the property of \emph{having no degree one edges}.

\begin{thm}
Let $M$ be an oriented three-manifold other than the lens space L(4,1). If $M$ is closed, let $\mathfrak{T}(M)$ denote the set of single vertex triangulations of $M$. If $M$ has boundary, let  $\mathfrak{T}(M)$ denote the set of ideal triangulations of $M$. Exclude from $\mathfrak{T}(M)$ any triangulations consisting of a single tetrahedron. Then the subgraph of $\mathfrak{T}(M)$ consisting of triangulations without degree one edges is connected under 2-3 and 3-2 moves.
\label{main theorem}
\end{thm}

The restrictions on the manifold being orientable, and the triangulation being either ideal or having only a single vertex are likely not very serious. Most probably, a few other special cases would need to be ruled out. For brevity however, in this paper we restrict to these cases. The restriction ruling out L(4,1) is a special case, similar to the cases of single-tetrahedron triangulations. Most of the triangulations of the manifold L(4,1) that do not have degree one edges are in one large connected component, but there is also an infinite family of isolated triangulations without degree one edges, for which any 2-3 move would introduce a degree one edge. These triangulations have no degree three edges, and so there are no 3-2 moves to consider. We discuss this family in Remark \ref{L41_triangs}.

The property of having no degree one edges is very weak, but is a prerequisite for a triangulation to have essential edges, or to be 1-efficient, or for most other properties of triangulations that have been investigated. We might hope that it will be possible to build up to connectivity of these stronger properties from weaker properties such as having no degree one edges.

\section{Definitions and preliminaries}\label{defns_and_prelims}

\begin{defn}
A \emph{model triangulation} is a set of
identical oriented $3$--simplices together with a set of
orientation-reversing face pairings.  A face cannot be glued to itself, and each face must be paired with another face. We refer to the quotient space after making the identifications given by the face pairings as a \emph{triangulation}. 

The simplices prior to identification, and their vertices, edges, and so on, are
called \emph{model cells}.  The map of a model cell to its image in
the triangulation need not be a homeomorphism -- this occurs whenever a cell is glued to itself in some fashion.

An \emph{ideal triangulation} is the result of removing the vertices from a triangulation, producing a manifold with boundary. For a closed manifold $M$, we define $\mathfrak{T}(M)$ to be the set of single-vertex triangulations. For a manifold $M$ with boundary, we define $\mathfrak{T}(M)$ to be the set of ideal triangulations of $M$, and we assume that $M$ has no spherical boundary components.
\end{defn}

\begin{rmk}
A triangulation as defined above (i.e. with material vertices, rather than an ideal triangulation) is topologically a manifold if and only if the \emph{link} of each vertex (i.e. the boundary of a small closed regular neighbourhood of the vertex) is a sphere.
\end{rmk}

\begin{defn}
The \emph{degree} of an edge $e$, written $\deg(e)$, is the number of model edges that are identified to form $e$.
\end{defn}

In this paper, we will translate freely back and forth between a triangulation and its dual \emph{special spine}. 

\begin{defn}
A \emph{spine} of a manifold $M$ with non-empty boundary is a compact subpolyhedron $P$, such that $M$ collapses to $P$. For a closed manifold $M$, a spine is a spine of the complement of an open ball in $M$. If we are given a triangulation of a manifold, we obtain the dual (special) spine by inserting into each tetrahedron a \emph{butterfly}. See Figure \ref{butterfly}. A spine is \emph{simple} if every point on a spine has a neighbourhood homeomorphic to one of the pictures shown in Figure \ref{simple_spine_points}. 

\begin{figure}[htb]
\centering
\subfloat[A butterfly inside of a tetrahedron.]{
\includegraphics[width=0.2\textwidth]{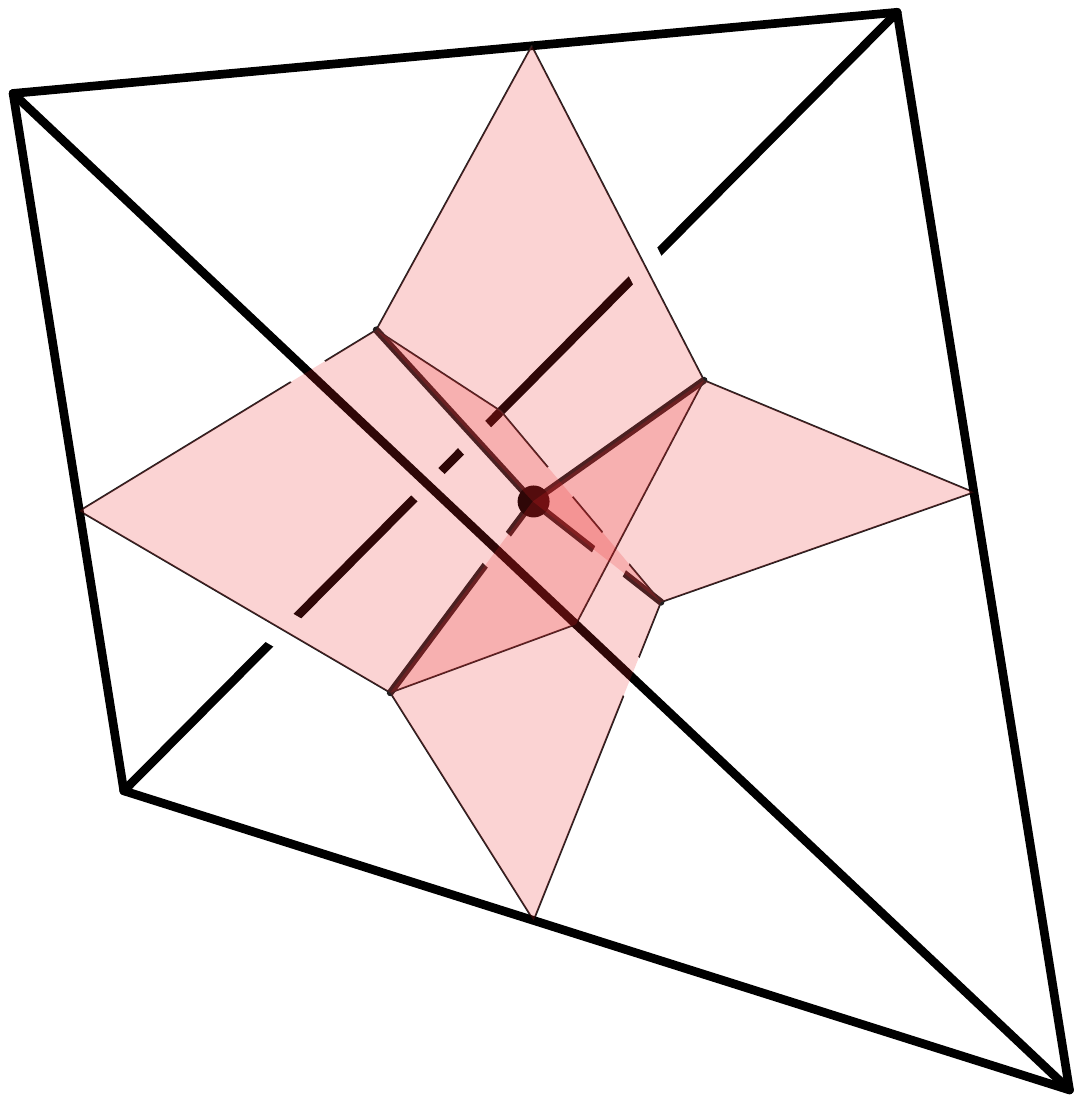}
\label{butterfly}
}
\quad
\subfloat[Three different types of point on a simple spine. From left to right: a non-singular point, a triple point, and a true vertex.]{
\includegraphics[width=0.65\textwidth]{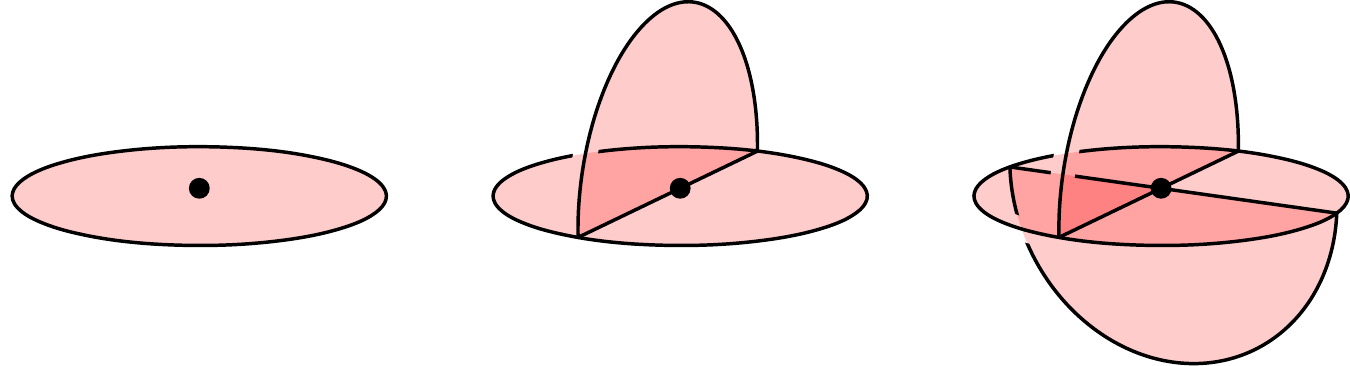}
\label{simple_spine_points}
}
\caption{} 
\label{spine_definitions}
\end{figure}

A connected component of the set of non-singular points is a \emph{2-stratum}. A connected component of the set of triple points is a \emph{1-stratum}. A simple spine is \emph{special} if each 2-stratum is an open disk and each 1-stratum is an open interval. 
\end{defn}
For further details on special spines and the moves on them described in this paper, see Matveev's book, \emph{Algorithmic topology and classification of 3-manifolds}~\cite{matveev}.

\begin{defn}
Let $\cT$ be a triangulation with at least two tetrahedra. A \emph{2-3 move} can be performed on any pair of distinct tetrahedra of $\cT$ that share a triangular face $\tri$. We remove $\tri$ and the two tetrahedra, and replace them with three tetrahedra arranged around a new edge, which has endpoints the two vertices not on $\tri$. See Figure \ref{2-3}. A \emph{3-2 move} is the reverse of a 2-3 move, and can be performed on any triangulation with a degree three edge $e$, where the three tetrahedra incident to $e$ are distinct (i.e. the three model edges that are identified to form $e$ are edges of distinct model tetrahedra).
\end{defn}

\begin{figure}[htb]

\centering
\subfloat[The 2-3 and 3-2 moves.]{
\labellist
\small\hair 2pt
\pinlabel 2-3 at 380 395
\pinlabel 3-2 at 380 314
\endlabellist
\includegraphics[width=0.44\textwidth]{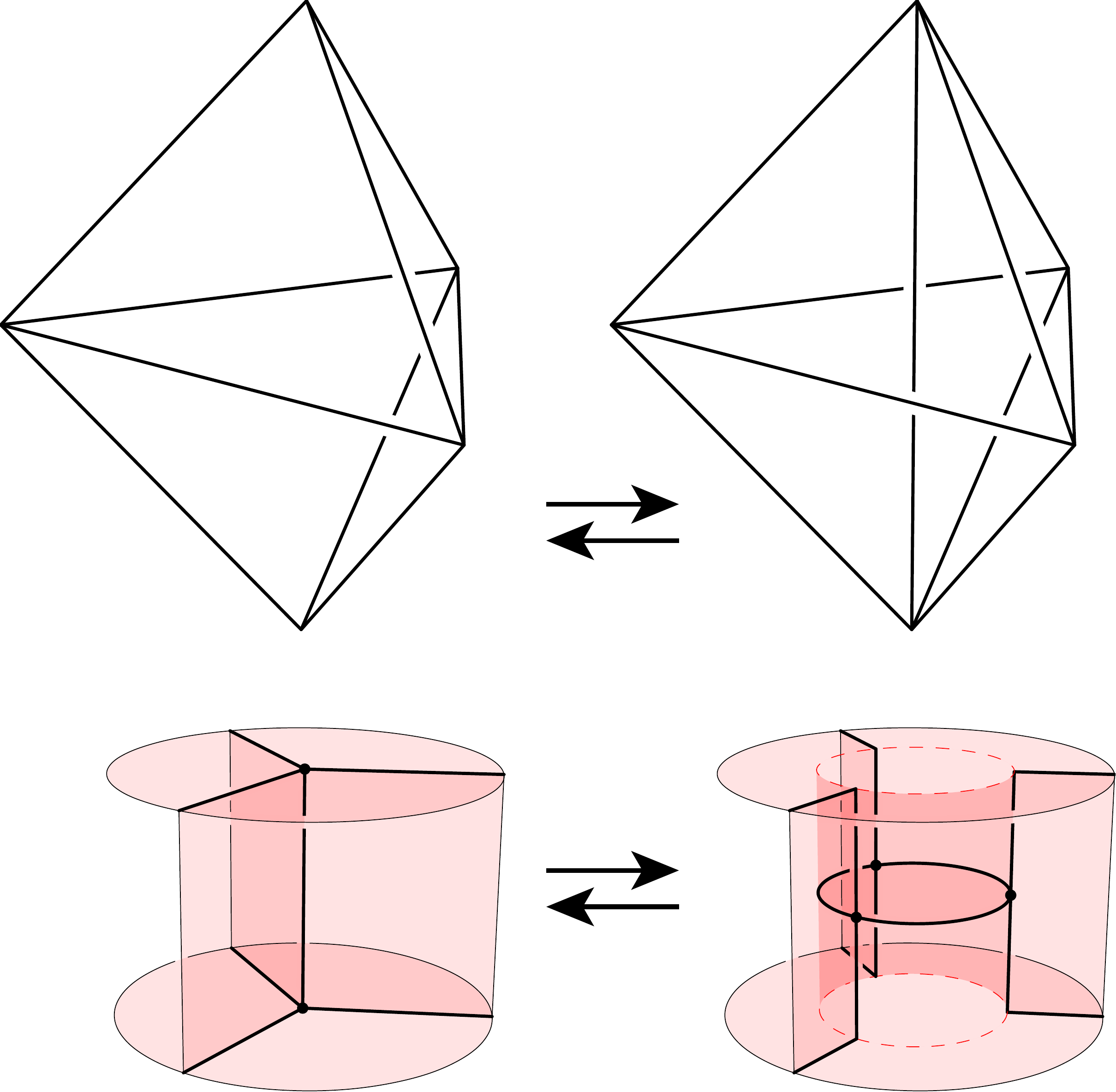}
\label{2-3}}
\qquad
\subfloat[The 0-2 and 2-0 moves.]{
\labellist
\small\hair 2pt
\pinlabel 0-2 at 339 539.5
\pinlabel 2-0 at 339 452.5
\endlabellist
\includegraphics[width=0.39\textwidth]{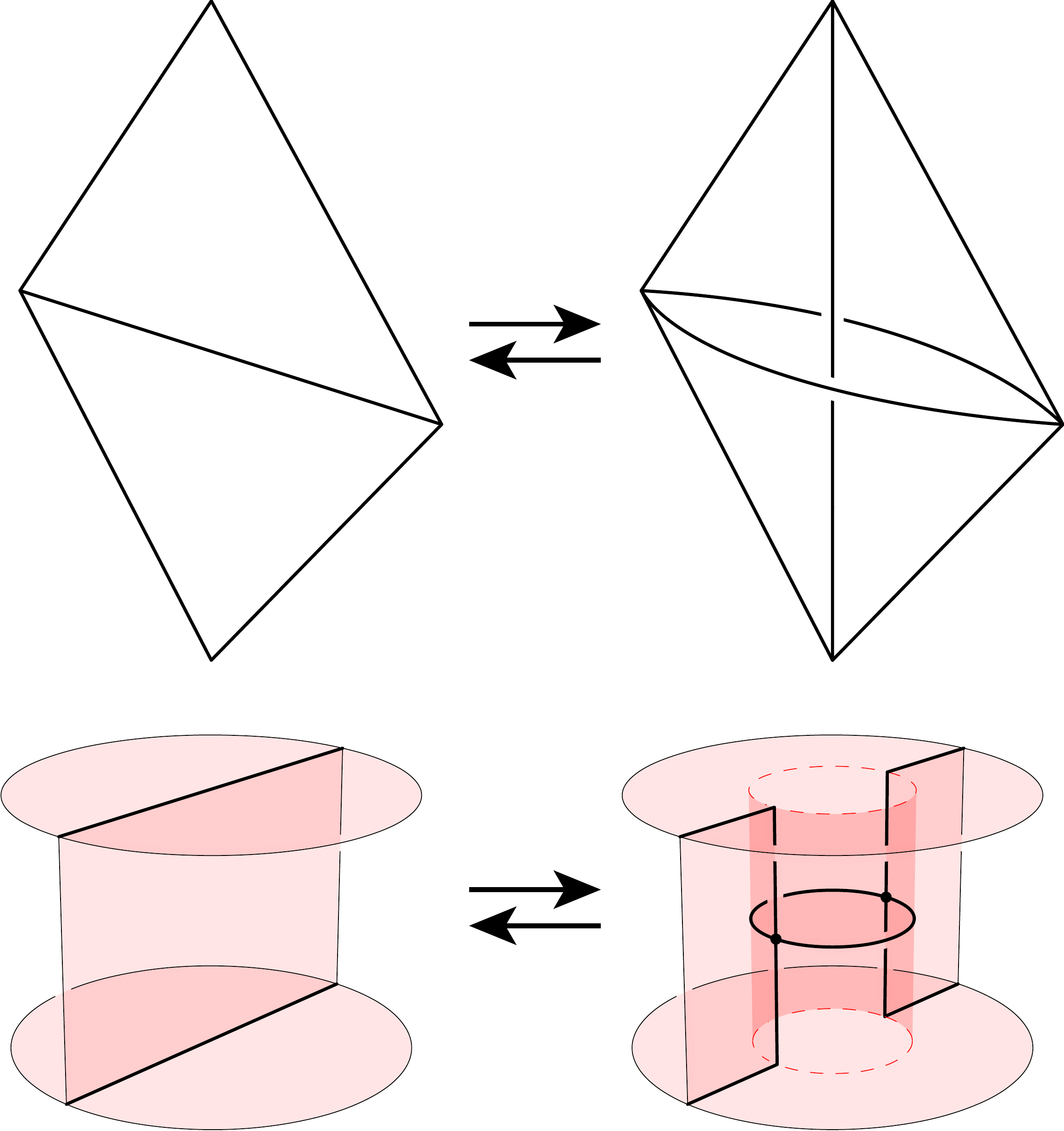}
\label{0-2}
}
\caption{Moves on (topological) triangulations, and their dual special spines.}
\label{2-3 and 0-2 moves}
\end{figure}

\begin{figure}[htb]
\centering
\labellist
\small\hair 2pt
\pinlabel $e$ at 121 145
\pinlabel 0-2 at 274 147
\endlabellist
\includegraphics[width=0.5\textwidth]{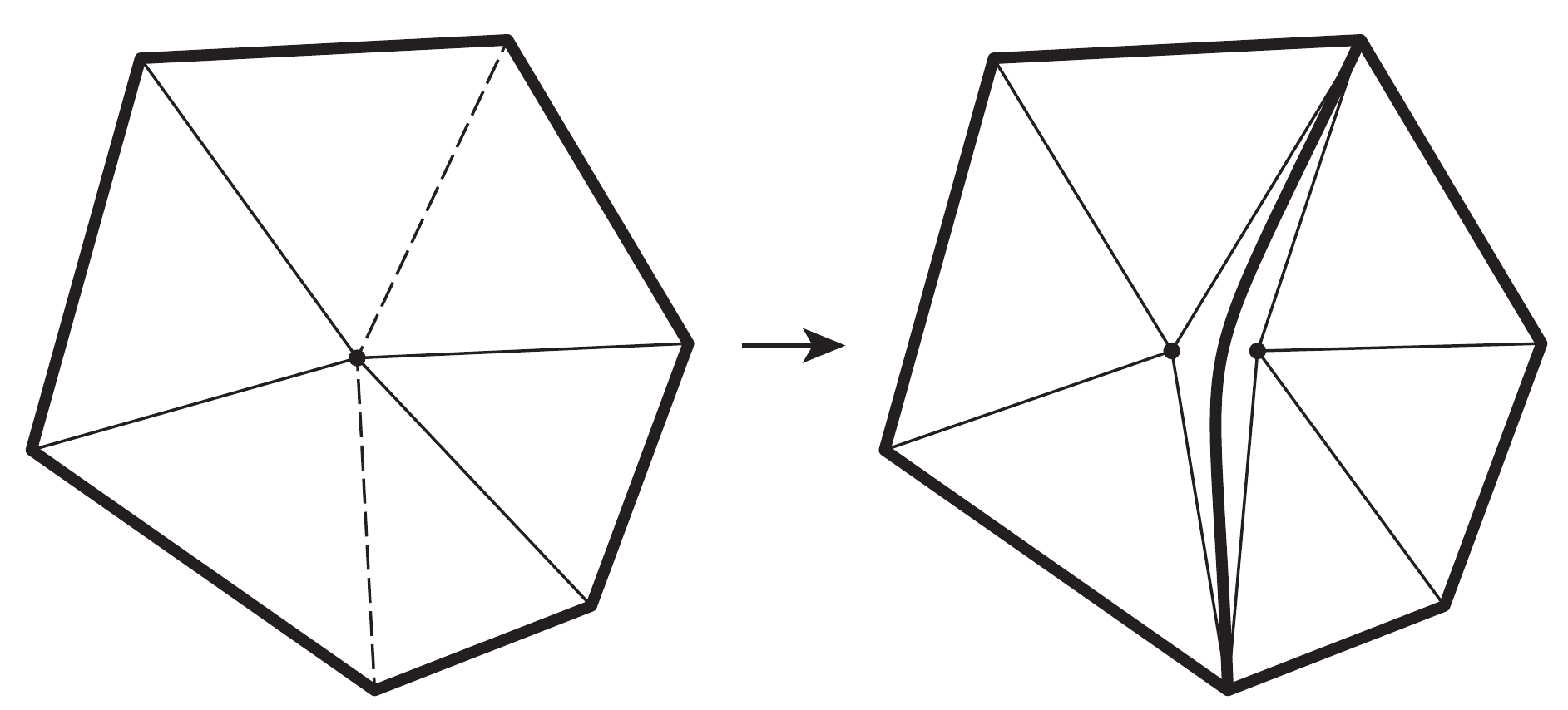}
\caption{The 0-2 move, shown acting on a cross-section through the book of tetrahedra around the edge $e$. Edges in the plane of the cross-section are drawn with a thick line, while intersections with triangular faces are drawn with a thin line. The intersections with the two triangular faces the move is performed on are drawn with a dashed line. One half book is on the left, the other on the right.}
\label{0-2_cross-section}
\end{figure}

The focus of this paper is on connectivity of triangulations under the 2-3 and 3-2 moves. In order to organise the argument we will also consider a number of other moves, including the 0-2 and 2-0 moves, and (later) the V-move.

\begin{defn}
\label{0-2_move}
Let $\cT$ be a triangulation. A \emph{0-2 move} can be performed on any pair of distinct triangular faces of $\cT$ that share an edge $e$. Around the edge $e$, the tetrahedra of $\cT$ are arranged in a cyclic sequence, which we call a \emph{book of tetrahedra}. (Note that tetrahedra may appear more than once in the book.) The two triangles and $e$ separate the book into two \emph{half--books}. We unglue the tetrahedra that are identified across the two triangles, duplicating the triangles and also duplicating $e$. We glue into the resulting hole a pair of tetrahedra glued to each other in such a way that there is a degree two edge between them. See Figure \ref{0-2}. In Figure \ref{0-2_cross-section} we see the effect of the 0-2 move on the book of tetrahedra around $e$. The figure shows a cross-section through the book, perpendicular to $e$. 
A \emph{2-0 move} is the reverse of a 0-2 move, and can be performed on any triangulation with a degree two edge, where the two tetrahedra incident to that edge are distinct, and the two edges opposite the degree two edge are not identified. 
\end{defn}

\begin{rmk}
The 0-2 move is also called the {\em lune} move in the dual language of special spines. 
\end{rmk}

\begin{thm}[Matveev~\cite{matveev, matveev_2-3_connectivity_paper}, Piergallini~\cite{piergallini}] \label{connectivity} 
Let $\cT$ and $\cT'$ be two triangulations of a manifold, both of which have at least two tetrahedra, and which have the same number of vertices as each other. Then $\cT$ and $\cT'$ are connected by a sequence of 2-3 and 3-2 moves.
\end{thm}

Our goal is to modify a path of triangulations given by the above theorem to avoid degree one edges.

\section{Anatomy of a degree one edge}

By definition, a degree one edge is formed from a single model edge $e$. The model tetrahedron that has $e$ as an edge has the two model faces incident to $e$ paired with each other. See Figure \ref{make_deg_1_edge}. Note that only one triangle of a triangulation containing a degree one edge is incident to that edge.

\begin{figure}[htb]
\centering
\labellist
\small\hair 2pt
\pinlabel $e$ at 251 39
\pinlabel $e'$ at 273.5 65
\endlabellist
\includegraphics[width=0.8\textwidth]{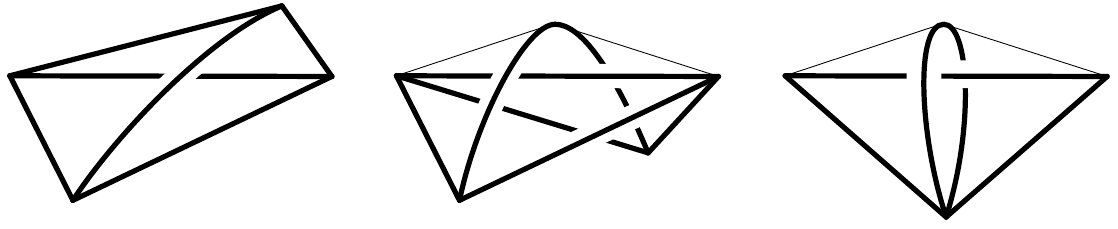}
\caption{A degree one edge is formed by gluing two faces of a tetrahedron to each other by ``closing the book'' around the edge. Edges of the tetrahedron are drawn with a thick line, while ``horizon'' lines of a triangular face that curves away from us are drawn with a thin line.} 
\label{make_deg_1_edge}
\end{figure}

In order to avoid introducing degree one edges when performing 2-3 and 3-2 moves, it will be useful to know under what conditions a degree one edge can arise. A 2-3 move alters the degrees of \emph{at most} nine edges of the triangulation: the nine edges of the two tetrahedra shown in the upper left of Figure \ref{2-3}. When a 2-3 move is performed, the degrees of the three ``equatorial'' edges go down by one, and the degrees of the other six edges go up by one. We say \emph{``at most} nine edges''  here, because there may be identifications among the edges, due to gluings not shown in Figure \ref{2-3}. Thus, in general, a 2-3 move can reduce the degree of an edge of the triangulation by up to three (if the three equatorial edges are identified with each other), while a 3-2 move can reduce the degree of an edge of the triangulation by up to six. However, any such large jumps down in degree cannot result in a degree one edge, as we will see in Corollary \ref{must_go_from_2_to_1}, which is derived from the following lemma.

\begin{lemma}\label{ways_to_create_deg_one}
There are two ways in which a degree one edge can be created, either via a 3-2 move or a 2-3 move. The two possibilities are shown in Figure \ref{deg2_to_deg1_edge}. In both cases, all of the tetrahedra shown are distinct (i.e. each tetrahedron shown corresponds to a distinct model tetrahedron). The triangulations before and after these moves may or may not have identifications amongst the boundary faces of the collections of tetrahedra shown.
\end{lemma}

\begin{figure}[htb]
\centering
\subfloat[The unique way in which a degree one edge can be formed via a 2-3 move.]{
\labellist
\small\hair 2pt
\pinlabel 2-3 at 110 51
\endlabellist
\includegraphics[width=0.45\textwidth]{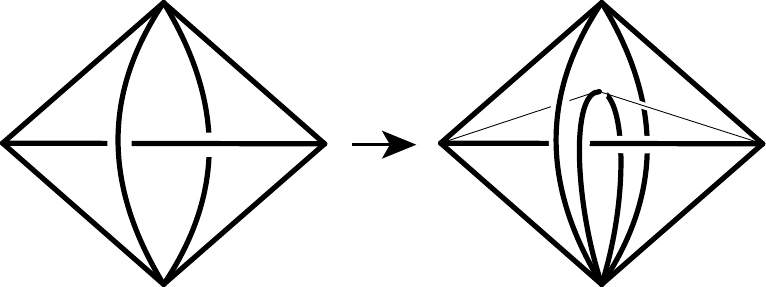}
\label{deg2_to_deg1_2-3}
}
\subfloat[The unique way in which a degree one edge can be formed via a 3-2 move.]{
\labellist
\small\hair 2pt
\pinlabel 3-2 at 110 51
\endlabellist
\includegraphics[width=0.45\textwidth]{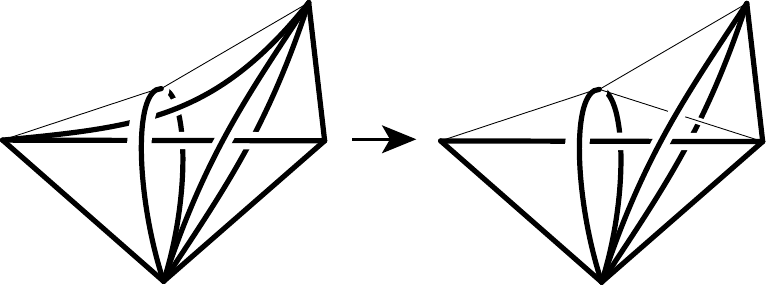}
\label{deg2_to_deg1_3-2}
}
\caption{} 
\label{deg2_to_deg1_edge}
\end{figure}

\begin{proof}
We consider the reverse move: starting with a tetrahedron $T$ incident to a degree one edge $e$, and considering which 2-3 or 3-2 moves can be applied to remove the degree one edge. 

First, consider the possible 3-2 moves that involve $T$ (and therefore can affect $e$). In Figure \ref{make_deg_1_edge}, $e$ is labelled, as is its opposite edge $e'$. A 3-2 move can only be applied to a degree three edge, all of whose incident tetrahedra are distinct (i.e. three distinct model tetrahedra have model edges identified to form the degree three edge). Therefore we cannot apply a 3-2 move to $e$, since its degree is incorrect, or to the other two edges incident to our tetrahedron, since the model tetrahedron has two of its model edges identified at each of these two edges. So the only possible edge that a 3-2 move can be applied to is $e'$. This move is shown in reverse in Figure \ref{deg2_to_deg1_2-3}. A 3-2 move results in two distinct tetrahedra, so we must start with two distinct tetrahedra when producing a degree one edge as in  Figure \ref{deg2_to_deg1_2-3}. 

Second, consider the possible 2-3 moves that involve $T$ (and therefore can affect $e$). Such a move must be performed on a triangle of the triangulation, whose two model triangles are on distinct model tetrahedra. Therefore we cannot use the one triangle incident to $e$, and must use one of the other two faces of $T$. Up to symmetry these give the same configuration. The reverse move is shown in Figure \ref{deg2_to_deg1_3-2}. The result of the 2-3 move consists of three distinct tetrahedra, so we must start with three distinct tetrahedra when producing a degree one edge as in Figure \ref{deg2_to_deg1_3-2}. 
\end{proof}

\begin{cor}\label{must_go_from_2_to_1}
A degree one edge comes from reducing the degree of a degree two edge by one. No single 2-3 or 3-2 move can convert any edge of degree higher than two into a degree one edge.
\end{cor}

\begin{proof}
Lemma \ref{ways_to_create_deg_one} lists the two ways to produce a degree one edge, and in both cases the edge in question has degree two beforehand.
\end{proof}

\begin{rmk}
It is possible that a 3-2 move can make two degree one edges at once. A 2-3 move can also make more than one degree one edge, but only in a triangulation with multiple material vertices, or both ideal and material vertices.
\end{rmk}

\section{Proof of the main theorem}

In this section, we prove Theorem \ref{main theorem}. We first describe the general strategy of the proof, with details to come in the subsequent pages.

We are given two triangulations $\cT$ and $\cT'$ of a manifold $M$, neither of which has a degree one edge, and we must find a path of 2-3 and 3-2 moves which connects from one to the other, and which does not pass through a triangulation with a degree one edge. Theorem \ref{connectivity} gives us the existence of a path $\gamma = (\cT = \cT_1, \ldots, \cT_N = \cT')$ connected by 2-3 and 3-2 moves, although some of the intermediate triangulations $\cT_i, 1<i<N$ may have degree one edges. Our plan is to modify the path to detour around any such triangulations.

First, observe that the tetrahedron $T$ incident to a degree one edge $e$ is unchanged for the entire lifetime of the degree one edge. Any 2-3 or 3-2 move that alters $T$ is one of the moves shown in Figure \ref{deg2_to_deg1_edge}, as we saw in the proof of Lemma \ref{ways_to_create_deg_one}. This means in particular that the triangle $\tri$ incident to $e$ is also unchanged for the lifetime of the degree one edge. 

Suppose at first that there are no degree one edges in the path $\gamma$ from $\cT_1$ to $\cT_i$. Suppose that a single degree one edge $e$ is introduced in the triangulation $\cT_{i+1}$, persists until $\cT_{j-1}$, and then is removed in $\cT_j$. Let the single triangle incident to $e$ be called $\tri$. We take a detour from $\cT_i$, performing a sequence of moves (which do not go through any triangulations with degree one edges) that results in a triangulation $\overline{\cT_i}$. This triangulation is identical to $\cT_i$, other than that the triangle $\tri$ has been unglued, and a triangulation of a three-ball with boundary consisting of two triangles is glued onto the two revealed faces of the triangulation. Said another way, we unglue the triangle $\tri$, and insert a \emph{triangular pillow}. Our triangular pillow must not have any degree one edges inside of it. The act of gluing it in where $\tri$ was increases the degree of the three edges incident to $\tri$.

We now continue the path $\gamma$, with $\overline{\cT_i}$ in place of $\cT_i$. The next move no longer produces a degree one edge, since the degree of $e$ has been increased by the insertion of the triangular pillow. As we continue this parallel path to the original $\gamma$, no moves alter the triangular pillow, by our previous observation about the fact that $\tri$ is unchanged for the lifetime of the degree one edge.

We continue, up until we get to the triangulation $\overline{\cT_j}$, which is the same as $\cT_j$ with $\tri$ unglued and the triangular pillow inserted into the resulting hole. After this move, we remove the triangular pillow by another sequence of moves, that converts $\overline{\cT_j}$ back into $\cT_j$ (this is precisely the reverse process to that of inserting the triangular pillow). Having completed our detour, we continue with the path $\gamma$, until we reach $\cT_N = \cT'$. The resulting path then has no triangulations with degree one edges.

If we can perform the above detour, then we can similarly deal with paths with multiple degree one edges, even multiple produced on the same move. All we need to do is apply the move of inserting a triangular pillow for each degree one edge in turn, leaving it there for the lifetime of the degree one edge, then remove it immediately afterwards.

Given this strategy, all we need to do is describe our triangular pillow, and show how it can be inserted into the triangle incident to an edge $e$ just before the degree of $e$ is to be reduced to one, all without introducing any degree one edges in the process.

\subsection{The triangular pillow}
\label{triangular_pillow}

First, we describe the triangulation we use for our triangular pillow. Figure \ref{make_bird_beak} shows a \emph{bird beak}, which is two tetrahedra arranged around a degree two edge. On the right we draw the bird beak in a suggestive manner -- rotating the upper half down so that it looks more like a real-life bird's beak. In Figure \ref{combine_two_beaks}, we glue together two bird beaks, interleaving the mandibles of each beak. The resulting triangulation of the three-ball has two triangular faces on its boundary, and has no degree one edges in its interior. Ungluing a triangle $\tri$ of a triangulation and inserting this triangular pillow adds three, three, and eight to the degrees of the edges incident to $\tri$. 

\begin{figure}[htb]
\centering
\includegraphics[width=0.5\textwidth]{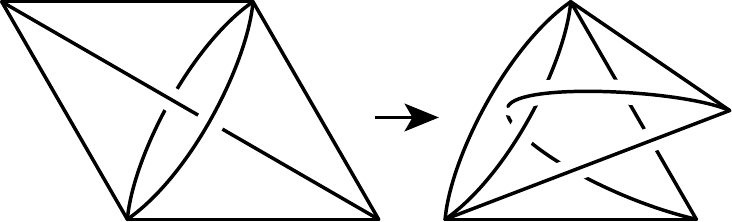}
\caption{A bird beak, is two tetrahedra arranged around a degree two edge.} 
\label{make_bird_beak}
\end{figure}

\begin{figure}[htb]
\centering
\labellist
\small\hair 2pt
\pinlabel $3$ at 200 36
\pinlabel $3$ at 238 36
\pinlabel $8$ at 219 9
\endlabellist
\includegraphics[width=0.65\textwidth]{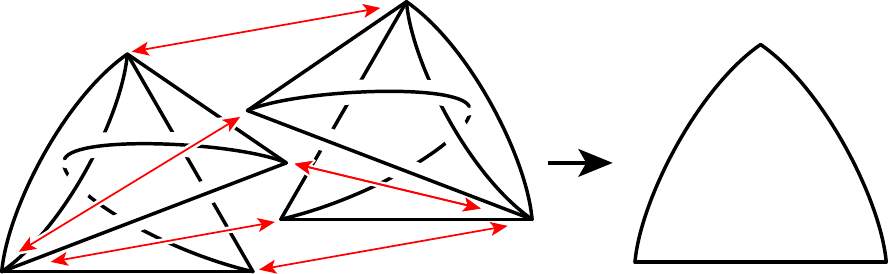}
\caption{Our triangulated triangular pillow is formed from two bird beaks, interleaved with each other. On the right, the contributions to the degrees of the three edges incident to the triangular pillow are shown. The four edges internal to the triangular pillow have degrees two, two, three and three.} 
\label{combine_two_beaks}
\end{figure}

\subsection{The V-move}

We must insert our triangular pillow using 2-3 and 3-2 moves. As a first step, we introduce a bird beak. We use Matveev's \emph{V-move}~\cite[Figure 1.13]{matveev}. The effect of the V-move is shown in Figure \ref{triangulation_V-move}: it wraps a bird beak around two faces of a tetrahedron. Note that there is a symmetry in the resulting three tetrahedra: we can also see this as wrapping a bird beak around the opposite two faces of the same tetrahedron. There are three different possible V-moves to apply in a tetrahedron, because there are three pairs of opposite edges.

\begin{figure}[htb]
\centering
\includegraphics[width=0.5\textwidth]{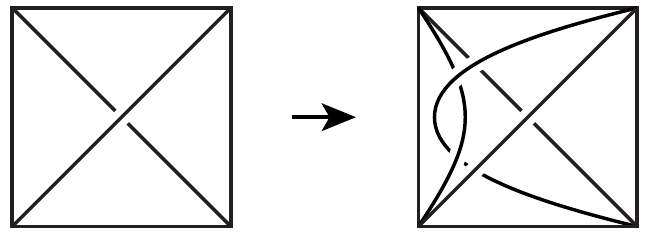}
\caption{The V-move wraps a bird beak around two faces of a tetrahedron. } 
\label{triangulation_V-move}
\end{figure}

\begin{figure}[htb]
\centering
\subfloat[Before the V-move.]{
\includegraphics[width=0.35\textwidth]{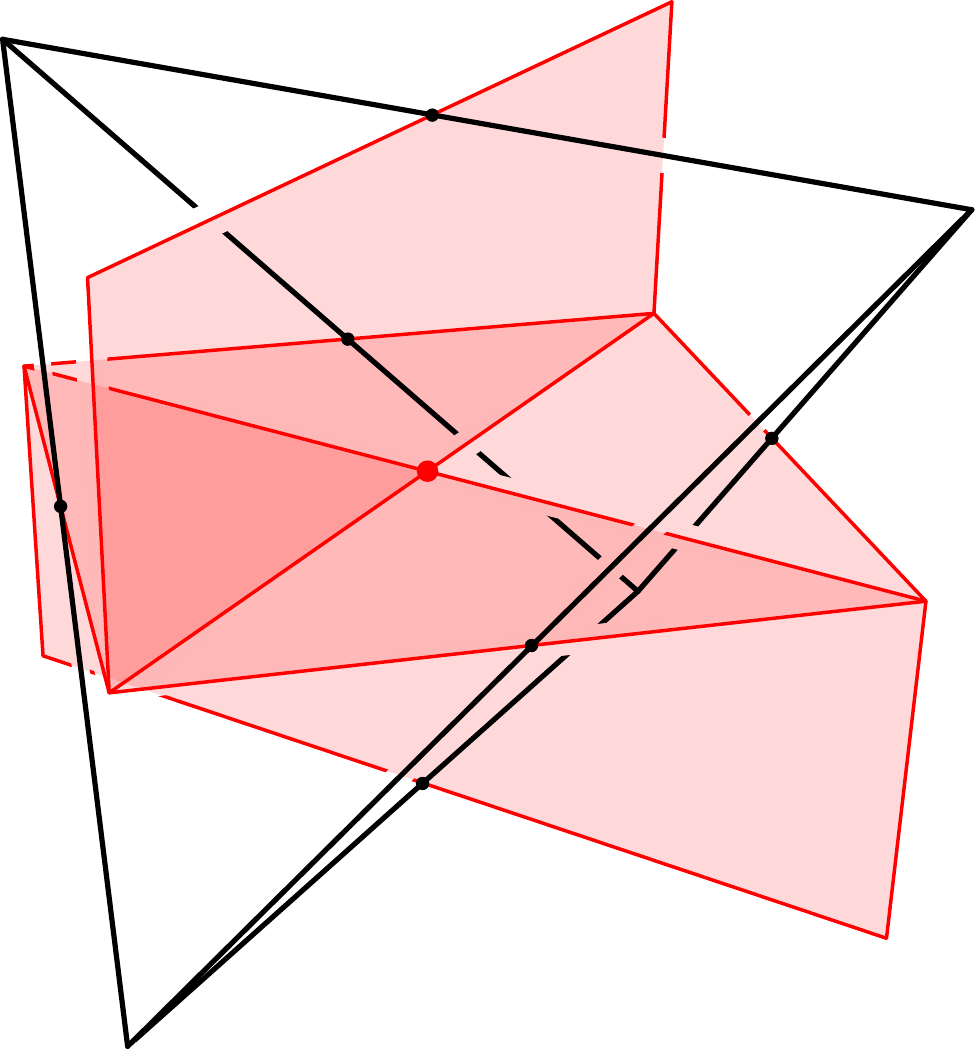}
}
\qquad
\subfloat[After the V-move.]{
\includegraphics[width=0.35\textwidth]{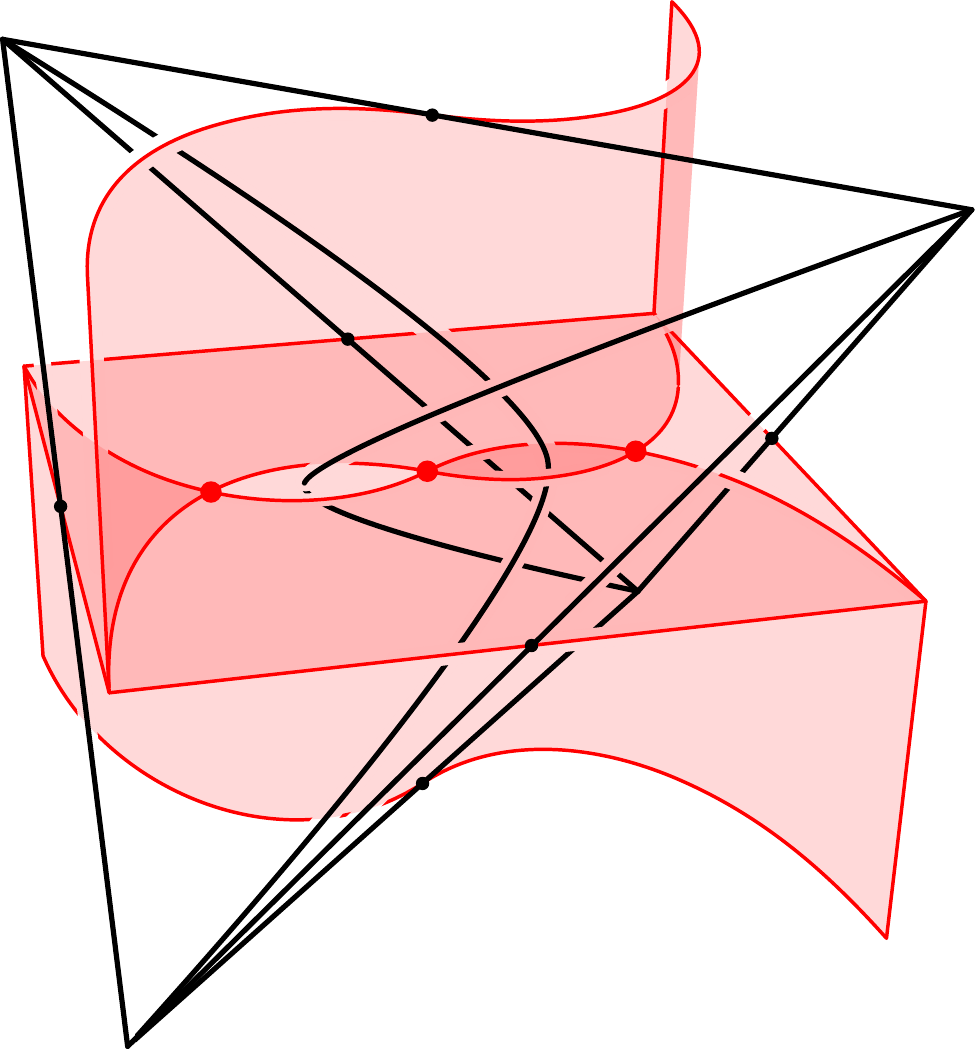}
}
\caption{The V-move, with both the triangulation and dual special spine shown.} 
\label{triangulation_and_spine_V-move}
\end{figure}

The V-move can be implemented using 2-3 and 3-2 moves, assuming that the triangulation has more than one tetrahedron. The process is easier to understand in the dual setting of special spines. Figure \ref{triangulation_and_spine_V-move} shows the V-move again, with both the triangulation and the special spine shown. Vertices, edges and faces of the dual spine correspond to tetrahedra, faces and edges of the triangulation.

\begin{lemma}
Let $\tri$ be a face of a triangulation which has no degree one edges. Suppose that $\tri$ is incident to two distinct tetrahedra $T_1, T_2$, and suppose that the edges of $\tri$ (which may not be distinct) all have degree at least three. Then we can perform any of the three V-moves in either $T_1$ or $T_2$ by sequences of 2-3 and 3-2 moves, none of which introduce a degree one edge at any point.
\label{V-move_lemma}
\end{lemma}

\begin{proof}
Figure \ref{V-move}, adapted from \cite[Figure 1.15]{matveev} shows the process of implementing the V-move with 2-3 and 3-2 moves, using the dual spine to the triangulation. By inspection, we can check that only the three faces marked with a $*$ have their degree reduced below their starting value during the process, and only by one, during the first 2-3 move. (In the dual picture, these three faces correspond to the three edges of the triangle $\tri$.) Of the new faces added, the smallest degree we see is two. Moreover, by Corollary \ref{must_go_from_2_to_1}, even if there were identifications between the faces marked $*$, the only way that such a face could become degree one is if it were degree two at the start. By assumption, this is not the case. 
\end{proof}

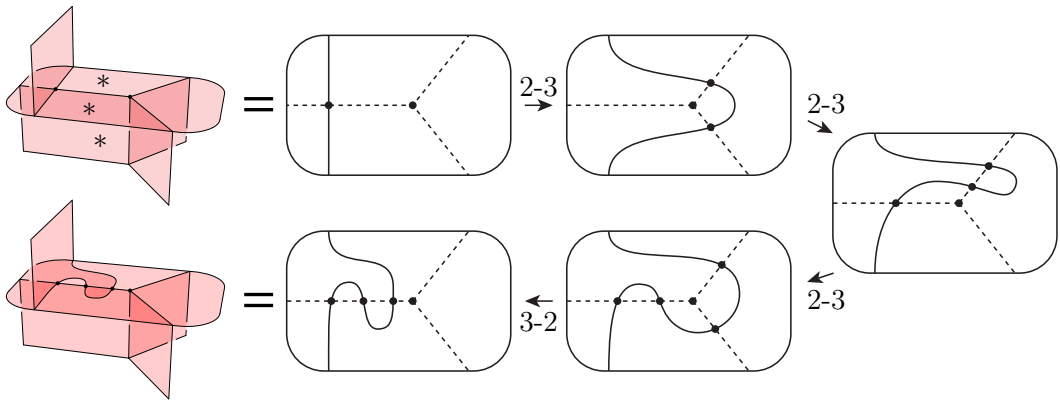
\begin{figure}[htb]
\centering
\labellist
\small\hair 2pt
\pinlabel $*$ at 65 205
\pinlabel $*$ at 56 187
\pinlabel $*$ at 63 166
\pinlabel 2-3 at 345 201
\pinlabel 2-3 at 530 190
\pinlabel 2-3 at 530 64
\pinlabel 3-2 at 345 50
\endlabellist
\includegraphics[width=0.9\textwidth]{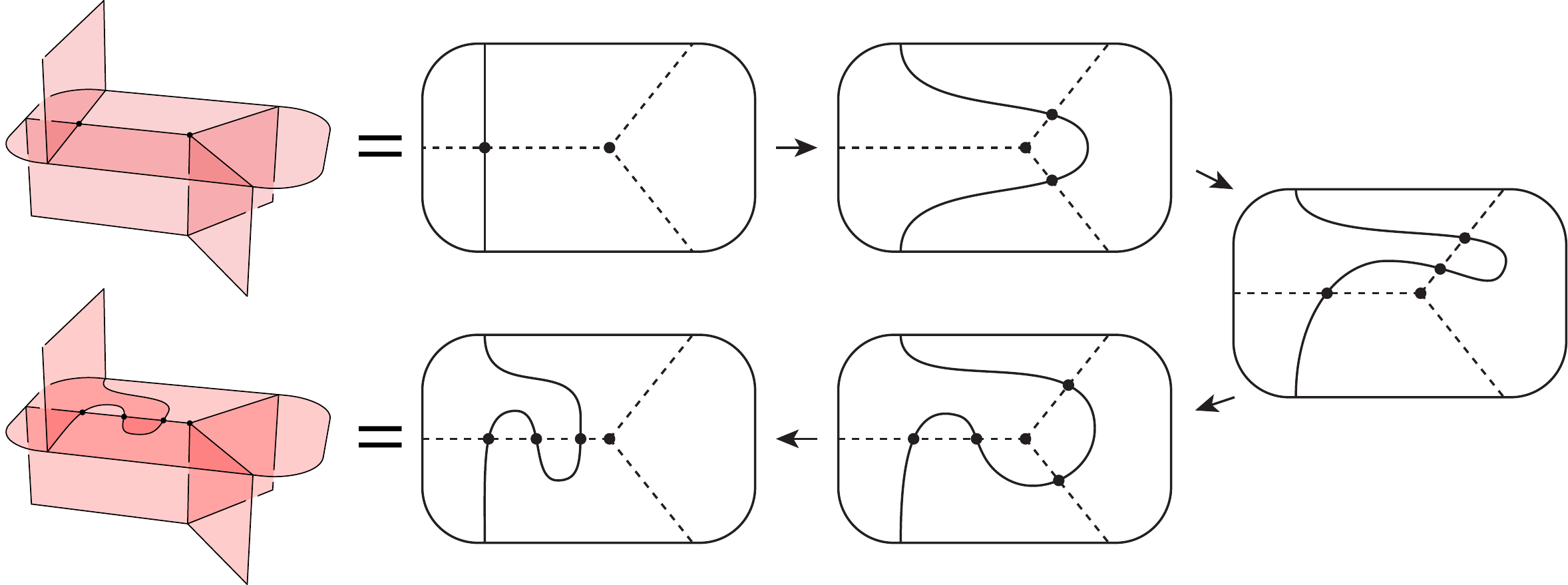}
\caption{The V-move is a composition of 2-3 and 3-2 moves.} 
\label{V-move}
\end{figure}

\subsection{Rotating the mandibles of a bird beak}

With the V-move, we introduce a bird beak, wrapped around two faces and across their common edge $e$ of a tetrahedron. We can think of the introduction of the bird beak as a 0-2 move (as in Definition \ref{0-2_move}), and so it splits the book of tetrahedra around $e$ into two half-books, one of which contains only the one tetrahedron we wrapped the bird beak around.

We will also need to be able to rotate the mandibles of the bird beak around in the split book of tetrahedra, effectively moving tetrahedra from one half-book to the other. This will let us close the mandibles on each other and make our triangular pillow. The move shown in Figure \ref{move_beak} achieves this goal, allowing us to move one mandible of the beak from one side of a tetrahedron to the other, assuming that this third tetrahedron is distinct from the two tetrahedra of the bird beak. The move is shown in the dual picture in Figure \ref{move_beak_spine}, which is adapted from \cite[Fig. 1.18]{matveev}. 

\begin{figure}[htb]
\subfloat[The move can be performed by applying a 2-3 move followed by a 3-2 move.]{
\centering
\labellist
\small\hair 2pt
\pinlabel 2-3 at 297 110
\pinlabel 3-2 at 631 110
\endlabellist
\includegraphics[width=0.9\textwidth]{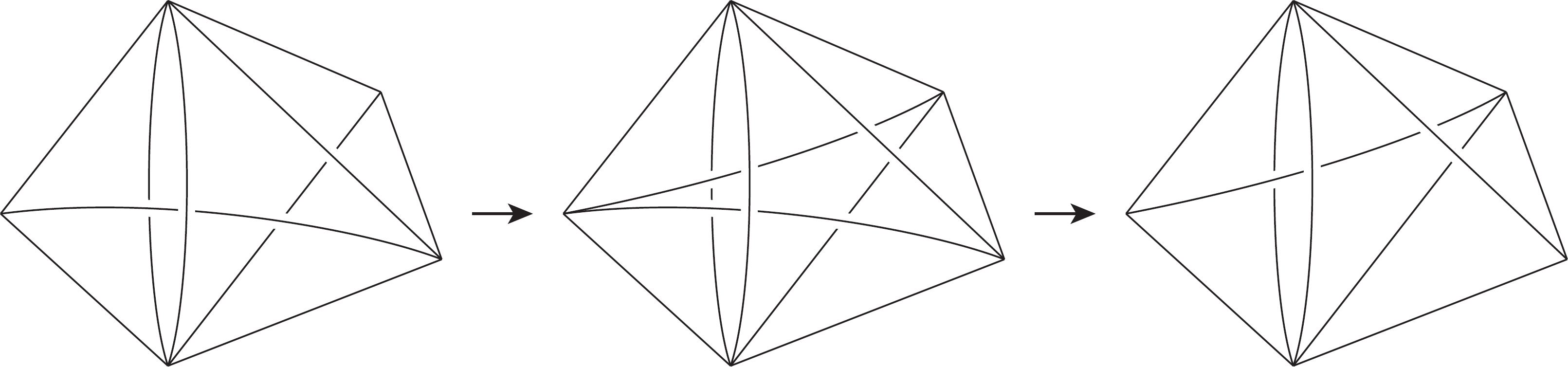}
}

\subfloat[The move shown in cross-section view. The lower mandible of the beak rotates to the right within the book of tetrahedra, moving the tetrahedron $T$ from the right half-book to the left half-book.]{
\centering
\labellist
\small\hair 2pt
\pinlabel $T$ at 153 55
\pinlabel $T$ at 440 50
\endlabellist
\includegraphics[width=0.5\textwidth]{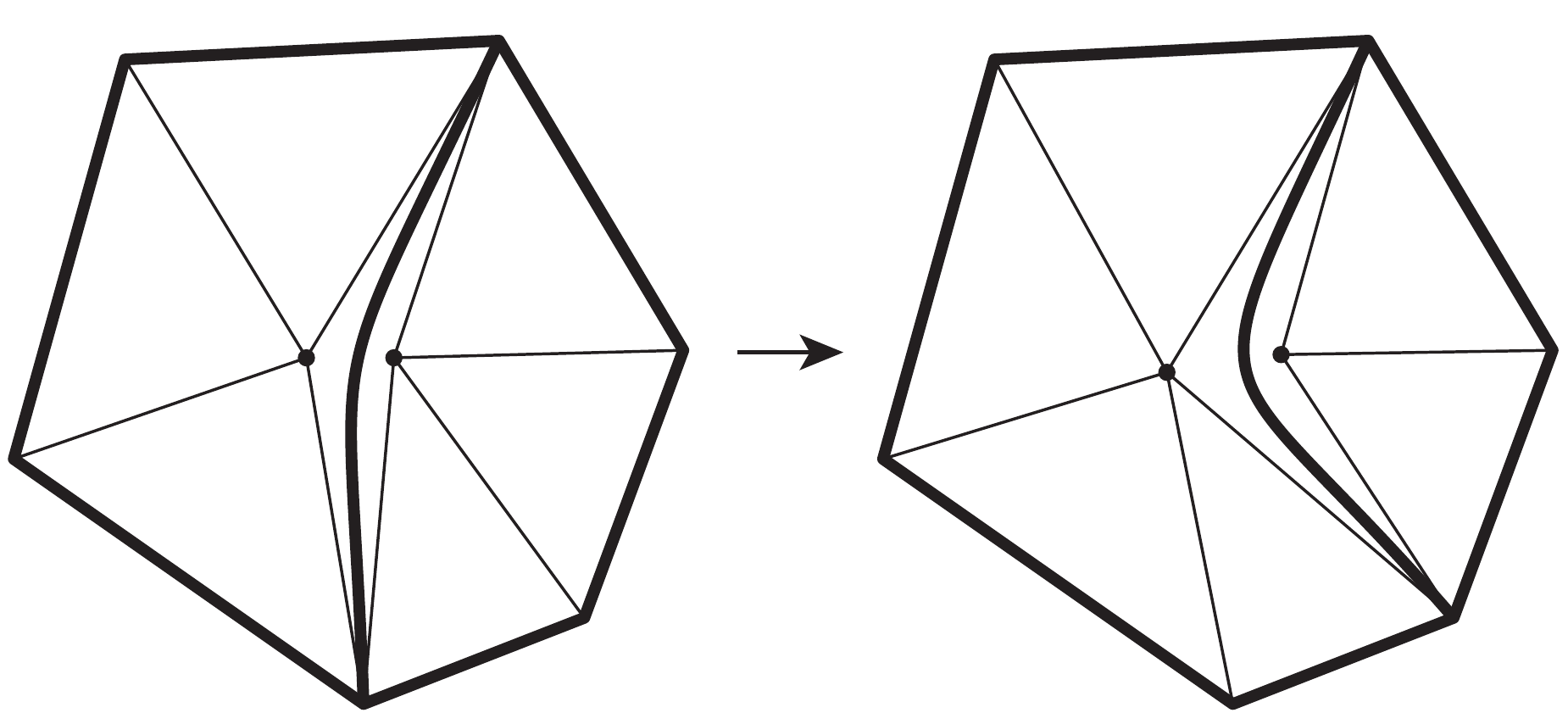}
}
\caption{Rotating a mandible of a bird beak past a tetrahedron.} 
\label{move_beak}
\end{figure}

\begin{rmk}
The move shown in Figure \ref{move_beak} does not allow us to rotate the two mandibles of a single bird beak \emph{past each other}. Here the configuration to consider involves a model triangle $\tri$, two of whose edges are identified to form a single edge $e$ of the triangulation. If we perform a 0-2 move, inserting a bird beak and splitting apart the book of tetrahedra around $e$, it can happen that the two mandibles of the beak are glued to each other across $\tri$, and we may want to swap their order. The move shown in \cite[Fig. 1.19]{matveev} allows us to rotate mandibles past each other in such a configuration, but we will not require such a move in this paper.
\label{move_beak_past_self}
\end{rmk}

Note that we do not want to rotate one mandible to close onto the other mandible of a single bird beak, as this would result in a degree one edge where the beak is folded over onto itself.

\begin{lemma}
Assume that we have a triangulation with a bird beak, and that collapsing the bird beak via a 2-0 move would not result in a triangulation with a degree one edge. Then
we can perform the move of rotating the bird beak past a tetrahedron that is not one of the two tetrahedra of the bird beak, by performing a 2-3 followed by a 3-2 move, neither of which introduce a degree one edge, unless this move results in a bird beak folded over onto itself.
\label{rotate_bird_beak}
\end{lemma}

\begin{proof}
See Figure \ref{move_beak_spine}. The faces marked $*$ have degree at least two, since otherwise we have a bird beak folded over onto itself. All other faces in the sequence above have degree at least two, either by directly counting vertices in the diagram, or because the corresponding face after performing the 2-0 move has degree at least two, by assumption.
\end{proof}

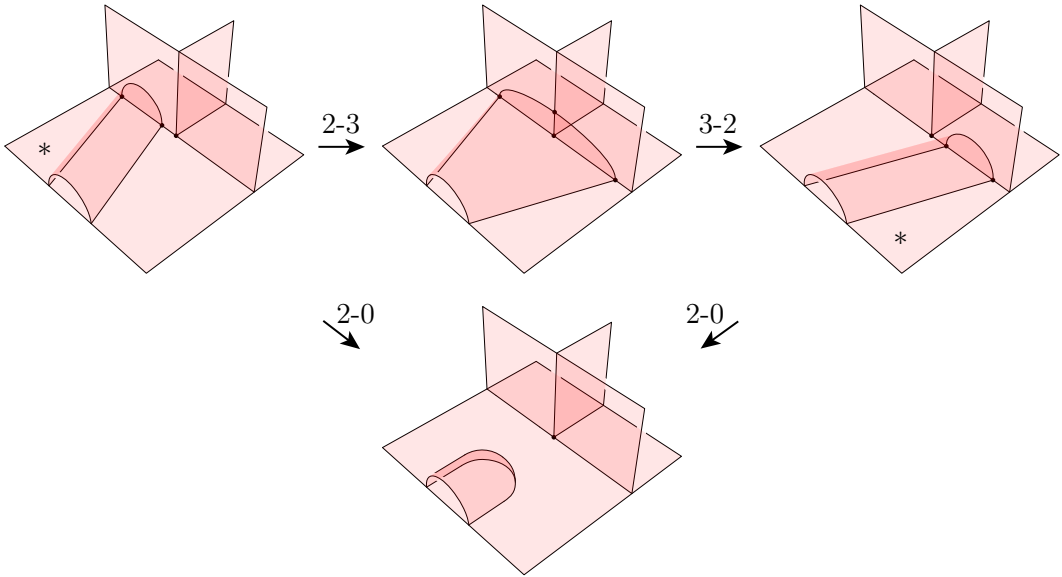
\begin{figure}[htb]
\centering
\labellist
\small\hair 2pt
\pinlabel 2-3 at 341 458
\pinlabel 3-2 at 723 458

\pinlabel 2-0 at 356 265
\pinlabel 2-0 at 710 265

\pinlabel $*$ at 41 432
\pinlabel $*$ at 908 341

\endlabellist
\includegraphics[width=0.9\textwidth]{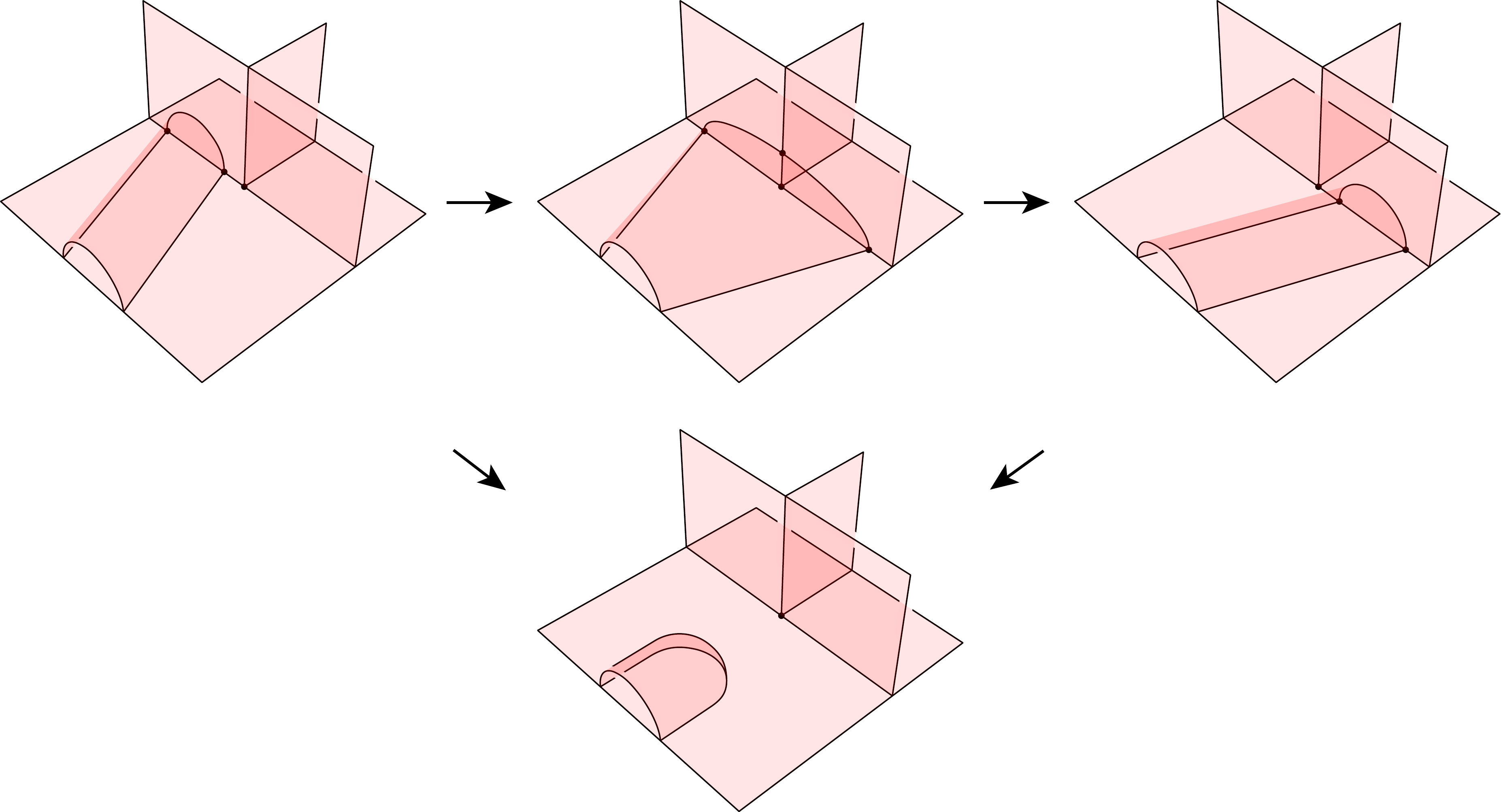}
\caption{The dual picture of rotating a bird beak past a tetrahedron.}
\label{move_beak_spine}
\end{figure}

If we start with a triangulation with no degree one edges and then perform a V-move, this lemma allows us to rotate the mandibles of the resulting bird beak to any position we wish without introducing any degree one edges, as long as we move the mandibles of the beak through a part of the book of tetrahedra for which all tetrahedra are distinct (to avoid the issues discussed in Remark \ref{move_beak_past_self}), and we do not close the beak onto itself.

\subsection{Inserting a triangular pillow}

We have our triangulation $\cT_i$, which has a triangle $\tri$, with incident edge $e$ which is degree two, and which in $\cT_{i+1}$ becomes degree one. Let the other two edges of $\tri$ be called $e'$ and $e''$, and suppose that the two tetrahedra incident to $e$ are also glued along the edges $\bar{e}'$ and $\bar{e}''$. See Figure \ref{near_tri_and_degree_2}.

\begin{figure}[htb]
\centering
\labellist
\small\hair 2pt
\pinlabel $\tri$ at 41 74
\pinlabel $e$ [t] at 45 58
\pinlabel $e'$ [bl] at 95 87
\pinlabel $\bar{e}'$ [tl] at 95 28
\pinlabel $f$ [l] at 69 34
\pinlabel $e''$ [br] at 35 87
\pinlabel $\bar{e}''$ [tr] at 35 28
\pinlabel $v$ [l] at 127 58
\endlabellist
\includegraphics[width=0.3\textwidth]{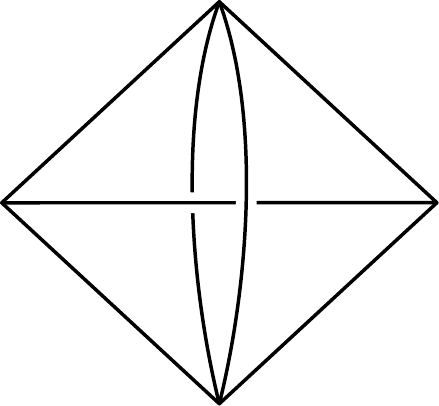}
\caption{The vicinity of the triangle $\tri$, with edge $e$ of degree two.} 
\label{near_tri_and_degree_2}
\end{figure}
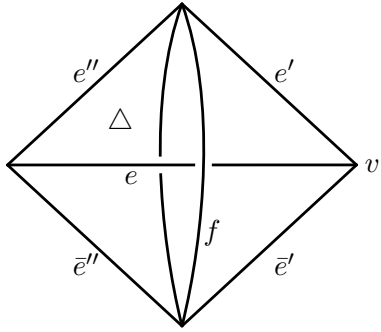

Our plan is to use two V-moves, on tetrahedra incident to each of $e'$ and $e''$, to introduce bird beaks that split the books of tetrahedra around $e'$ and $e''$. We then use Lemma \ref{rotate_bird_beak} to rotate the mandibles of the two beaks around until they form our triangular pillow.
In order to apply Lemma \ref{V-move_lemma} and so apply the V-move, splitting the book of tetrahedra around $e'$ say,  we need to find a triangle incident to $e'$, all of whose edges are degree at least three. In particular, we require that $e'$ has degree at least three -- we first check that this is the case.

Suppose not, and assume that $e'$ has degree two. Then considering Figure \ref{near_tri_and_degree_2}, we see that the triangulation would have a vertex $v$ with spherical link, with only the two tetrahedra incident to $e$ incident to it. But there must then be another vertex of the triangulation, since not all vertices of the two tetrahedra are at $v$. In the closed case this contradicts the assumption in Theorem \ref{main theorem} that the triangulation has only a single vertex. In the case that $M$ has boundary, the spherical link of $v$ gives another contradiction, since we assume that $M$ has no spherical boundary components.

So we know that $e'$ has degree at least three, as do $e'', \bar{e}',$ and $\bar{e}''$ by the same argument. Next, we need to find a triangle incident to $e'$, all of whose edges have degree at least three. Consider Figure \ref{near_tri_and_degree_2} again. We know that the edge $\bar{e}'$ has degree at least three. If the edge $f$ also has degree at least three then we are done: we have found a triangle incident to $e'$ with three edges of degree at least three. If not, then $f$ must have degree two (since there are no degree one edges), and $f$ is incident to another tetrahedron, $T$, say, which is incident to $e', \bar{e}', e'', \bar{e}''$ and one other edge, say $f_1$, which is opposite $f$ in $T$. The tetrahedron $T$ has two triangles that are incident to $e'$, one of which we have already checked, and the other, which has edges $e', e'',$ and $f_1$. We saw above that $\bar{e}''$ has degree at least three, and so if $f_1$ also has degree three, then we are again done. If not, then again $f_1$ must have degree two, and we continue with another tetrahedron incident to $e', \bar{e}', e'', \bar{e}''$ and one other edge, say $f_2$, which is opposite $f_1$ in the new tetrahedron. 

We continue in this fashion, building a stack of tetrahedra, until we either find a triangle incident to $e'$ and two other edges of degree at least three, or the stack wraps around to join onto the back side of the two tetrahedra we started with, and with the final degree two edge being the edge opposite $e$ in the back tetrahedron shown in Figure \ref{near_tri_and_degree_2}. Note that this is the only way in which the stack of tetrahedra can glue to itself. 

Since the manifold is orientable, there are four cases to consider, depending on the angle by which the front of the stack of tetrahedra we have built is glued onto the back. 
\begin{itemize}
\item If the stack is glued with no rotation (and so there must be an even number of tetrahedra in the stack), then the four vertices of the tetrahedra are distinct after gluing with spherical vertex links, which again contradicts the hypotheses given in Theorem \ref{main theorem}.
\item If the stack is glued with a half-turn rotation (and so again there must be an even number of tetrahedra in the stack), then there are two distinct vertices after gluing, each with spherical vertex links, and again we contradict the hypotheses given in Theorem \ref{main theorem}.
\item If the stack is glued with either a quarter turn or a three quarters turn, then we have a single-vertex triangulation of the manifold L(4,1), which is the exception listed in Theorem \ref{main theorem}.
\end{itemize}

\begin{rmk}
The above argument constructs the triangulations of L(4,1) mentioned in the introduction. Each is formed from a stack of an odd number of tetrahedra, with degree two edges between each pair of tetrahedra in the stack. The two triangles at the top of the stack are glued to the two triangles at the bottom of the stack, with a quarter (or equivalently, three-quarter) turn. No edge has degree three, so no 3-2 move is possible. Moreover, every possible 2-3 move introduces a degree one edge.
\label{L41_triangs}
\end{rmk}

So, we find a triangle incident to $e'$ for which all edges have degree at least three. Then using Lemma \ref{V-move_lemma}, we perform a V-move that results in a bird beak that splits apart the book of tetrahedra around $e'$. Figure \ref{first_V-move} shows the result in the case that the edge $f$ has degree at least three -- we get a bird beak wrapped around one of the two tetrahedra incident to $e$. Otherwise, we obtain a bird beak wrapped around a tetrahedron somewhere further up the stack. By using Lemma \ref{rotate_bird_beak}, we can rotate the mandibles of such a beak one by one down the stack (noting that all tetrahedra of the stack are distinct), until it is in the position shown in Figure \ref{first_V-move}. By moving the lower mandible before the upper one, we can ensure that we do not close the bird beak onto itself when rotating its mandibles, which would produce a degree one edge.
\begin{figure}[htb]
\centering
\labellist
\small\hair 2pt
\pinlabel $e$ [t] at 45 58
\pinlabel $\bar{e}'$ [tl] at 95 28
\pinlabel $f$ [l] at 69 34
\pinlabel $e''$ [br] at 35 87
\pinlabel $\bar{e}''$ [tr] at 35 28
\endlabellist
\subfloat[After performing the V-move, assuming that $f$ has degree at least three.]{
\includegraphics[width=0.4\textwidth]{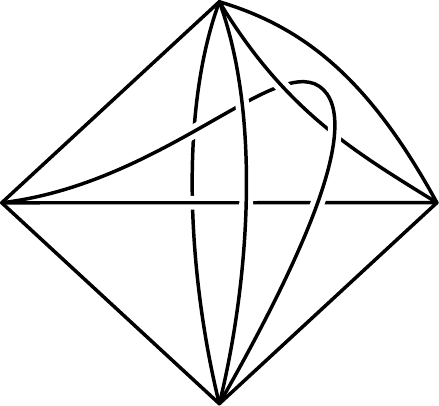}
\label{first_V-move}
}
\qquad
\subfloat[After performing the second V-move.]{
\labellist
\small\hair 2pt
\pinlabel $e$ [t] at 45 58
\pinlabel $\bar{e}'$ [tl] at 95 28
\pinlabel $f$ [l] at 69 34
\pinlabel $\bar{e}''$ [tr] at 35 28
\endlabellist
\includegraphics[width=0.4\textwidth]{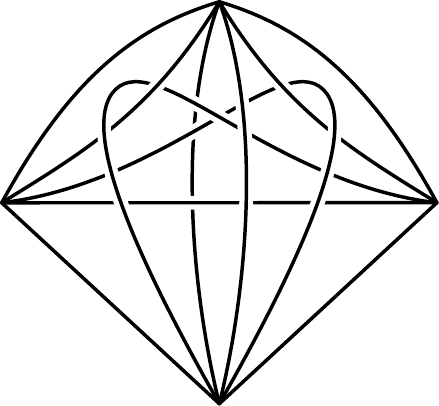}
\label{second_V-move}
}
\caption{V-moves used to make two bird beaks.} 
\label{V-move_in_place}
\end{figure}
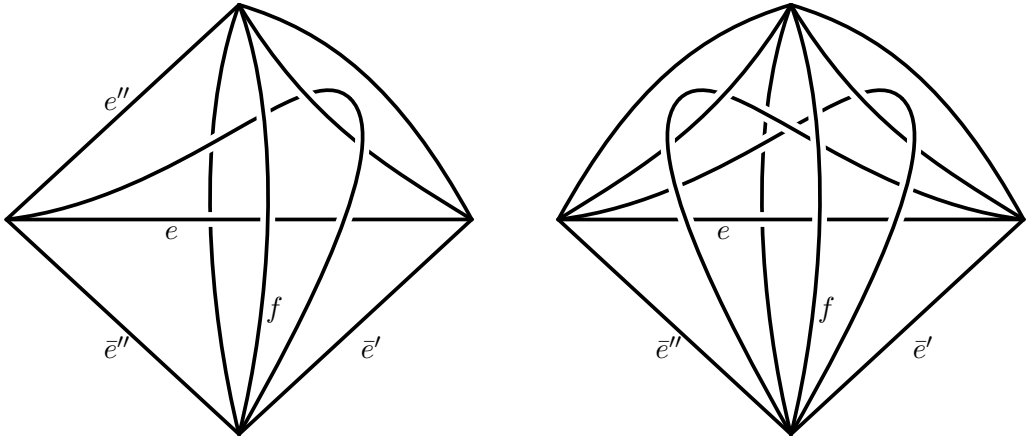

Next, we need to add the other bird beak needed to make our triangular pillow. As before, we know that edges $e''$ and $\bar{e}''$ have degree at least three, and after our previous moves, $f$ now has degree at least three as well. So we can apply Lemma \ref{V-move_lemma} here, splitting apart the book of tetrahedra around $e''$ and obtaining Figure \ref{second_V-move}. Finally, using Lemma \ref{rotate_bird_beak}, we close the two bird beaks around each others mandibles, producing our triangular pillow. See Figure \ref{close_bird_beaks}.

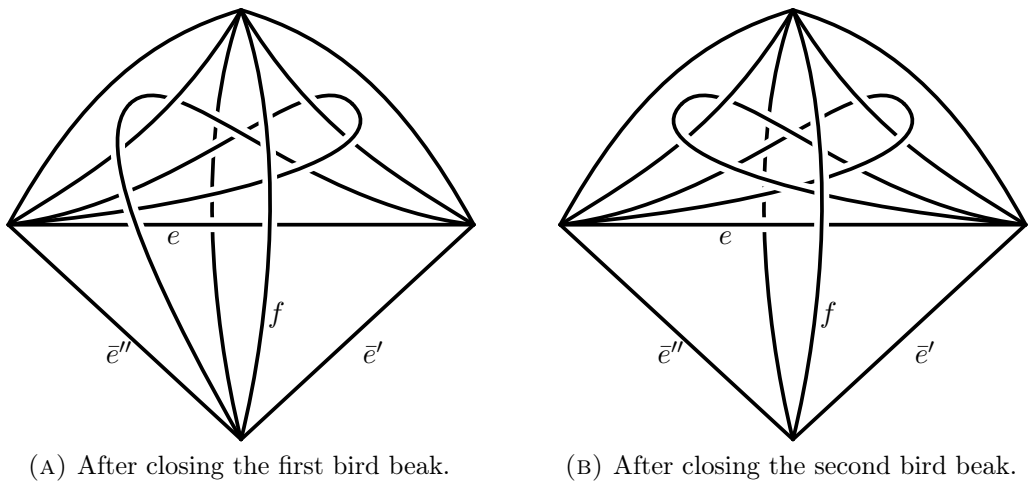
\begin{figure}[htb]
\centering
\labellist
\small\hair 2pt
\pinlabel $e$ [t] at 45 58
\pinlabel $\bar{e}'$ [tl] at 95 28
\pinlabel $f$ [l] at 69 34
\pinlabel $\bar{e}''$ [tr] at 35 28
\endlabellist
\subfloat[After closing the first bird beak.]{
\includegraphics[width=0.4\textwidth]{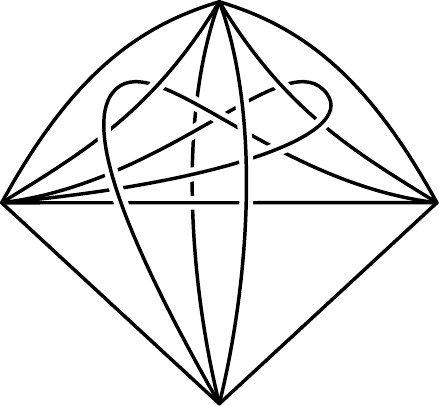}
\label{close_first_beak}
}
\qquad
\subfloat[After closing the second bird beak.]{
\labellist
\small\hair 2pt
\pinlabel $e$ [t] at 45 58
\pinlabel $\bar{e}'$ [tl] at 95 28
\pinlabel $f$ [l] at 69 34
\pinlabel $\bar{e}''$ [tr] at 35 28
\endlabellist
\includegraphics[width=0.4\textwidth]{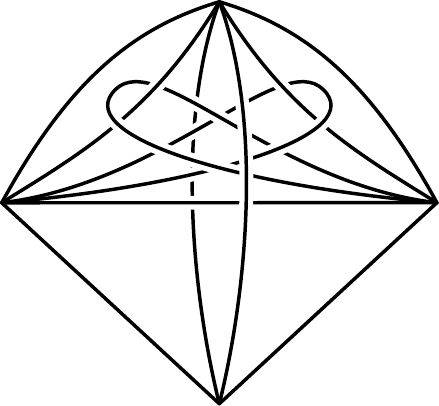}
\label{close_second_beak}
}
\caption{Closing the bird beaks.} 
\label{close_bird_beaks}
\end{figure}

To summarise, in this section we described a procedure to introduce a triangular pillow without introducing a degree one edge at any stage. This then completes the proof of Theorem \ref{main theorem}.

\bibliographystyle{../hamsplain}
\bibliography{../henrybib} 
\end{document}